\documentclass[twocolumn]{autart}    
\usepackage{amsmath,amssymb}                                     
\usepackage{algpseudocode} 
\usepackage{float} 
\usepackage{algorithm}
\usepackage[utf8]{inputenc} 
\usepackage[T1]{fontenc}    
\usepackage{amsmath} 
\usepackage{graphicx}
\usepackage{booktabs}
\newtheorem{theorem}{\bf{Theorem}}

\newtheorem{assumption}{Assumption}

\newtheorem{lemma}{\bf{Lemma}}

\newtheorem{remark}{\bf{Remark}}

\makeatletter
\providecommand{\qedsymbol}{$\blacksquare$}  
\newenvironment{proof}[1][Proof]{%
\par\noindent\textit{#1.}\ }{%
\hfill\qedsymbol\par}
\makeatother
\begin{document}

\begin{frontmatter}

\title{Random Reshuffling-Based Distributed Nash Equilibrium Seeking } 

\thanks[footnoteinfo]{This work was supported in part by the Joint Funds of the National
Natural Science Foundation of China under Grant U24A20258, in
part by the National Natural Science Foundation of China under
Grants 92367205, in part by the Zhejiang Provincial Natural Science
Foundation of China under Grant LRG25F030001, in part by the
funding of Leading Innovative and Entrepreneur Team Introduction
Program of Zhejiang under Grant 2023R01006, in part by the
Fundamental Research Funds for the Provincial Universities of
Zhejiang under Grant RF-C2023007.
Jun Hu, Chao Sun, Bo Chen, Jianzheng Wang, Zheming Wang
are with the Department of Automation, Zhejiang University of
Technology, Hangzhou 310023, China. Jun Hu, Chao Sun, Bo Chen,
Jianzheng Wang, Zheming Wang are also with Zhejiang Key Laboratory of Intelligent Perception and Control for Complex Systems,
Hangzhou 310023, China.\\
Corresponding author: Bo Chen (bchen@aliyun.com).}

\author[ZJUT,ZJ]{Jun Hu}
\author[ZJUT,ZJ]{Chao Sun}
\author[ZJUT,ZJ]{Bo Chen}
\author[ZJUT,ZJ]{Jianzheng Wang}
\author[ZJUT,ZJ]{Zheming Wang}

\address[ZJUT]{Department of Automation, Zhejiang University of Technology, Hangzhou 310023, China}
\address[ZJ]{hejiang Key Laboratory of Intelligent Perception and Control for Complex Systems, Hangzhou 310023, China}
          

\begin{abstract}
This paper studies distributed Nash equilibrium seeking for 
stochastic games under partial-decision information. Each player has access
only to its local sample functions and maintains local estimates of the other
players' actions through neighbor communication. We propose a random-reshuffling
stochastic pseudo-gradient scheme in which each player visits all local
components exactly once per epoch and performs projection only at the end of
the epoch. Compared with with-replacement stochastic approximation, random
reshuffling removes the epoch-start sampling noise, but the movement of the
inner iterates and the use of local estimates introduce additional drift terms.
To handle these effects, we develop a coupled decision-consensus error analysis
that jointly controls the reshuffling-induced inner drift and the estimate
disagreement. Under strong monotonicity and local Lipschitz regularity of the
pseudo-gradient mapping, the proposed method converges linearly to an
\(O(\alpha^2)\) mean-square neighborhood of the Nash equilibrium with constant
stepsizes. With diminishing stepsizes, exact mean-square convergence is
established. Numerical examples on networked  games show that random
reshuffling achieves smaller terminal errors than with-replacement stochastic
updates under matched component-gradient and communication budgets.
\end{abstract}

\end{frontmatter}

\section{Introduction}

Game theory provides a framework for modeling strategic interactions among multiple self-interested decision makers.
Among the solution concepts for noncooperative games, Nash equilibrium (NE) is one of the most fundamental, as it characterizes operating points at which no player can unilaterally improve its own objective. NE seeking has found wide applications in 
communication networks \cite{alpcan2002cdma}, smart grids \cite{5628271}, resource allocation \cite{1532399}, and multi-robot coordination \cite{10.1007/978-3-030-05816-6_16}, where multiple agents are coupled through their decision variables and cost functions while pursuing individual objectives.
In many modern engineering systems, however, these agents are distributed via a network and can communicate only with  neighbors. As a result, each player may not have direct access to the full action profile of all the others, which naturally leads to the problem of distributed NE seeking under partial-decision information.\par

To solve such networked games, distributed Nash equilibrium seeking has been
widely studied under different information, uncertainty, and communication
settings. Under partial-decision information, \cite{salehisadaghiani2019admm} proposed an ADMM-based method in which
players exchange local action estimates with their neighbors. \cite{Pavel2020_TAC} developed a doubly-augmented operator-splitting framework
for seeking variational generalized Nash equilibria with fixed stepsizes.
More recently, \cite{meng2023linear} established linear convergence
of a distributed gradient-tracking method for multi-cluster games under
partial-decision information.

Distributed Nash equilibrium seeking has also been investigated under
stochastic observations, uncertain agent dynamics, and restricted
communication. \cite{franci2022stochastic} developed
distributed stochastic-approximation and variance-reduced methods for
generalized Nash equilibrium seeking under partial-decision information. \cite{huang2022distributed} proposed a distributed stochastic Nash
equilibrium learning method for locally coupled network games with unknown
parameters. \cite{feng2023adaptively} studied adaptive distributed
Nash equilibrium seeking for uncertain heterogeneous multi-agent systems,
whereas \cite{chen2022quantization} investigated distributed Nash
equilibrium seeking under quantized communication. In addition, \cite{pang2020distributed} developed a consensus-based gradient-free method
requiring only local cost-function evaluations. Moreover, \cite{Ye2017Consensus} developed a consensus-based distributed Nash equilibrium seeking algorithm under partial-decision information, where each player estimates the actions of other players through local communication and updates its own action based on the estimated joint strategy.
\cite{Pang2021LimitedCost} further studied distributed Nash equilibrium seeking with limited cost-function knowledge and proposed a consensus-based gradient-free method that enables agents to seek the Nash equilibrium using only local cost measurements rather than explicit gradient information.
\par

Random Reshuffling (RR) is the without-replacement counterpart of SGD \cite{haochen2019random}, \cite{shamir2016without}. At the beginning of each epoch, RR generates a fresh permutation of the component indices and processes every component exactly once in the resulting order, thereby avoiding repeated revisits of already-sampled components within the same pass \cite{emmanouilidis2024stochastic}, \cite{shamir2016without}. Existing optimization results have shown that RR can outperform with-replacement SGD after a finite number of epochs \cite{haochen2019random} and, under constant step-sizes, drive the iterates to a smaller steady-state neighborhood than uniform with-replacement sampling \cite{ying2018stochastic}. More recently, in  variational-inequality problems, stochastic extragradient with RR has been shown to attain arbitrary accuracy without the large-batch requirements often needed by classical with-replacement schemes \cite{emmanouilidis2024stochastic}. In addition, RR has also demonstrated favorable convergence behavior in distributed optimization over networks \cite{huang2023distributed}. These results suggest that RR is a promising sampling mechanism for distributed stochastic Nash equilibrium seeking, and naturally raise the question of whether its advantages can be similarly exploited in distributed networked games.
\par
However, the role of RR in distributed Nash equilibrium seeking
remains largely unexplored, especially under partial-decision information.
Existing distributed stochastic NE methods mainly rely on independently sampled
stochastic approximation updates, with primary emphasis on pseudo-gradient
design, local estimation, and communication mechanisms. In contrast, the effect
of without-replacement sampling on coupled pseudo-gradient dynamics and network
disagreement has received much less attention. This motivates the present work,
where we study whether RR can serve as a principled sampling
mechanism for distributed stochastic NE seeking and improve solution accuracy
over conventional with-replacement SGD schemes.

To the best of our knowledge, RR has not been systematically
studied in distributed stochastic Nash equilibrium seeking. The main
contributions of this paper are threefold.

(1) We develop a RR-based distributed Nash equilibrium seeking
scheme under partial-decision information. Each player independently reshuffles
its local component functions, updates its decision using local estimates of the
joint action, and exchanges information only with its neighbors. Unlike
with-replacement stochastic pseudo-gradient methods, the proposed scheme visits
each local component exactly once per epoch and performs projection only at the
end of each epoch.

(2) We establish an epoch-wise error decomposition that separates the two
coupled perturbations arising in the proposed scheme: the reshuffling-induced
inner drift caused by within-epoch decision updates and the disagreement error
caused by evaluating component pseudo-gradients at local estimates rather than
at the true joint action. To accommodate epoch-end projection, we introduce a
compact local domain containing the auxiliary inner trajectories and show that
the required pseudo-gradient boundedness follows from local Lipschitz continuity
on this domain. We further derive an epoch-level consensus-error recursion that
quantifies the interaction between reshuffling and network disagreement.

(3) We establish convergence guarantees for both constant and diminishing
stepsizes. Under constant stepsizes, the decision and consensus errors converge
to \(O(\alpha^2)\) mean-square neighborhoods. Under diminishing stepsizes
satisfying standard summability conditions, both errors converge to zero in
mean square, yielding exact convergence to the Nash equilibrium. Numerical
experiments demonstrate that the proposed RR scheme achieves improved terminal
accuracy over its with-replacement counterpart under matched oracle complexity.\par

\noindent\textbf{Notation.}
Throughout this paper, \(0\) denotes either the real number zero or a zero
vector with an appropriate dimension. Let \([N]:=\{1,\ldots,N\}\) and
\([m]:=\{1,\ldots,m\}\). The symbols \(\mathbb R\) and \(\mathbb R^N\)
denote the set of real numbers and the \(N\)-dimensional Euclidean space,
respectively. All vectors are column vectors. The vector \(\mathbf 1_N\)
denotes the \(N\)-dimensional all-one vector, and \(I_N\) denotes the
\(N\times N\) identity matrix.
The Euclidean norm and inner product are denoted by \(\|\cdot\|\) and
\(\langle\cdot,\cdot\rangle\), respectively. The absolute value of a scalar is
denoted by \(|\cdot|\). \(\lambda_{\min}(\cdot)\) and
\(\lambda_{\max}(\cdot)\) denote its smallest and largest eigenvalues of a matrix,
respectively. The Kronecker product is denoted by \(\otimes\). For a closed
convex set \(\Omega\), \(\mathbb P_{\Omega}[\cdot]\) denotes the Euclidean
projection onto \(\Omega\), and \(\operatorname{dist}(x,\Omega)\) denotes the
distance from \(x\) to \(\Omega\). The expectation operator is denoted by
\(\mathbb E[\cdot]\). For a differentiable function \(f\), 
\(\nabla_{x_i}f(x)\) denotes the gradient of \(f\) with respect to \(x_i\)
evaluated at \(x\).
\section{Problem Formulation}
\subsection{Noncooperative Game and Nash Equilibrium}
Consider a Noncooperative game with $N$ players, where each player $i$ aims to
solve the following expected-value optimization problem:
\begin{align}
\min_{x_i\in\Omega_i} \ J_i(x_i,x_{-i})
:= \mathbb{E}_{\xi_i}\!\left[f_i(x_i,x_{-i};\xi_i)\right].\label{problem1}
\end{align}

In practice, the expected objective is approximated by the following
sample-average form:
\begin{align}
\min_{x_i\in\Omega_i} \ f_i(x_i,x_{-i}),
\qquad
f_i(x_i,x_{-i})
:= \frac{1}{m}\sum_{\ell=1}^{m}f_i(x_i,x_{-i};\xi_i^\ell).
\label{problem}
\end{align}
where $x_i\in\Omega_i\subset\mathbb{R}$ is the action of player $i$, and
$x_{-i}\in\Omega_{-i}\subset\mathbb{R}^{N-1}$ collects the actions of all
players except player $i$. Here, $\xi_i$ denotes a local random variable
associated with player $i$, and $\{\xi_i^\ell\}_{\ell=1}^m$ is a collection of
local samples (or scenarios) drawn from its distribution. For simplicity, each
player is assumed to have the same number $m$ of local samples. Equivalently,
by defining$f_i(x_i,x_{-i};\ell):=f_i(x_i,x_{-i};\xi_i^\ell)$,
the sample-average objective in \eqref{problem} can be rewritten as
\[
f_i(x_i,x_{-i})=\frac{1}{m}\sum_{\ell=1}^{m} f_i(x_i,x_{-i};\ell),
\]
where $\ell\in\{1,\ldots,m\}$ indexes the local component functions. Each
component may correspond to one historical sample, one simulated realization,
or one mini-batch available to player $i$.

\begin{remark}
The sample-average formulation in \eqref{problem} is a standard and
practically meaningful approximation of the expected-value game
\eqref{problem1}. Such sample-average approximation (SAA) schemes have
been widely used in stochastic optimization and stochastic equilibrium
problems\cite{kleywegt2002sample,xuu2013stochastic}. In this paper, we focus on the Nash equilibrium of the approximating
finite sum game \eqref{problem}. Under standard SAA consistency conditions,
the equilibrium of \eqref{problem} is expected to approach that of the
underlying expected-value game \eqref{problem1} as the sample size
increases, although this statistical approximation issue is not the main focus
here. Instead, our emphasis is on the algorithmic aspect: the finite sum
structure in \eqref{problem} is particularly suitable for the design and
analysis of RR, and it also motivates the sample-driven
benchmarks considered in Section~4.
\end{remark}
To guarantee the existence and uniqueness of the Nash equilibrium of problem \eqref{problem}, we next state several standard assumptions,
\begin{assumption}
$\Omega:=\Omega_1\times\Omega_2\times\dots\Omega_N$ is nonempty, convex and compact.\label{set}
\end{assumption}
\begin{assumption}
For each $i\in\{1,\dots,N\}$ and $\ell\in\{1,\dots,m\}$, the function $f_i(x;\ell)$ is continuously differentiable for $x$ in $\Omega$. In addition, for any fixed $x_{-i}$, $f_i(x_i,x_{-i};\ell)$ is convex in $x_i$, and $\nabla_{x_i} f_i(x;\ell)$ is $L$-Lipschitz continuous in $x_i$ over $\Omega_i$.
\label{fi}
\end{assumption}
\begin{assumption}
  The pseudo-gradient mapping $\nabla F(x):=[\nabla_{x_1} f_1(x), \nabla_{x_2} f_2(x),\dots \nabla_{x_N} f_N(x)]^\mathbf{T}\in\mathbb{R}^N$ is strongly monotone with modulus $\mu$ i.e. $\langle\nabla F(x)-\nabla F(y), x-y \rangle \ge \mu\|x-y\|^2$.\label{mapping}
\end{assumption}
Under Assumptions \ref{set}, \ref{fi} and \ref{mapping}, there exists a unique Nash equilibrium $\mathbf{x}_{\star}$\cite[theorem3]{DBLP:journals/corr/abs-1212-6235}.
\subsection{Random Reshuffling(RR)}
Unlike existing methods for solving \eqref{problem}, which typically rely on a with-replacement sampling scheme, we instead employ an incremental strategy based on RR. In RR, the local component functions are accessed sequentially according to a random permutation within each epoch.

Specifically, at the beginning of epoch $k$, each player $i$ generates a random permutation $\{\pi_0^i,\pi_1^i,\dots,\pi_{m-1}^i\}$ of the index set $\{1,2,\dots,m\}$. Then, along this permuted order, player $i$ performs $m$ successive inner updates:
\begin{align*}
&x_{i,k}^{\ell+1}
=
x_{i,k}^{\ell}
-
\alpha_k
\nabla_{x_i} f_i(\mathbf{x}_k^\ell;\pi_\ell^i),\\
&\text{where}\quad i\in[n],\ \ell=0,1,\dots,m-1.
\end{align*}
\begin{figure}[h]
    \centering
    \includegraphics[width=1.0\linewidth]{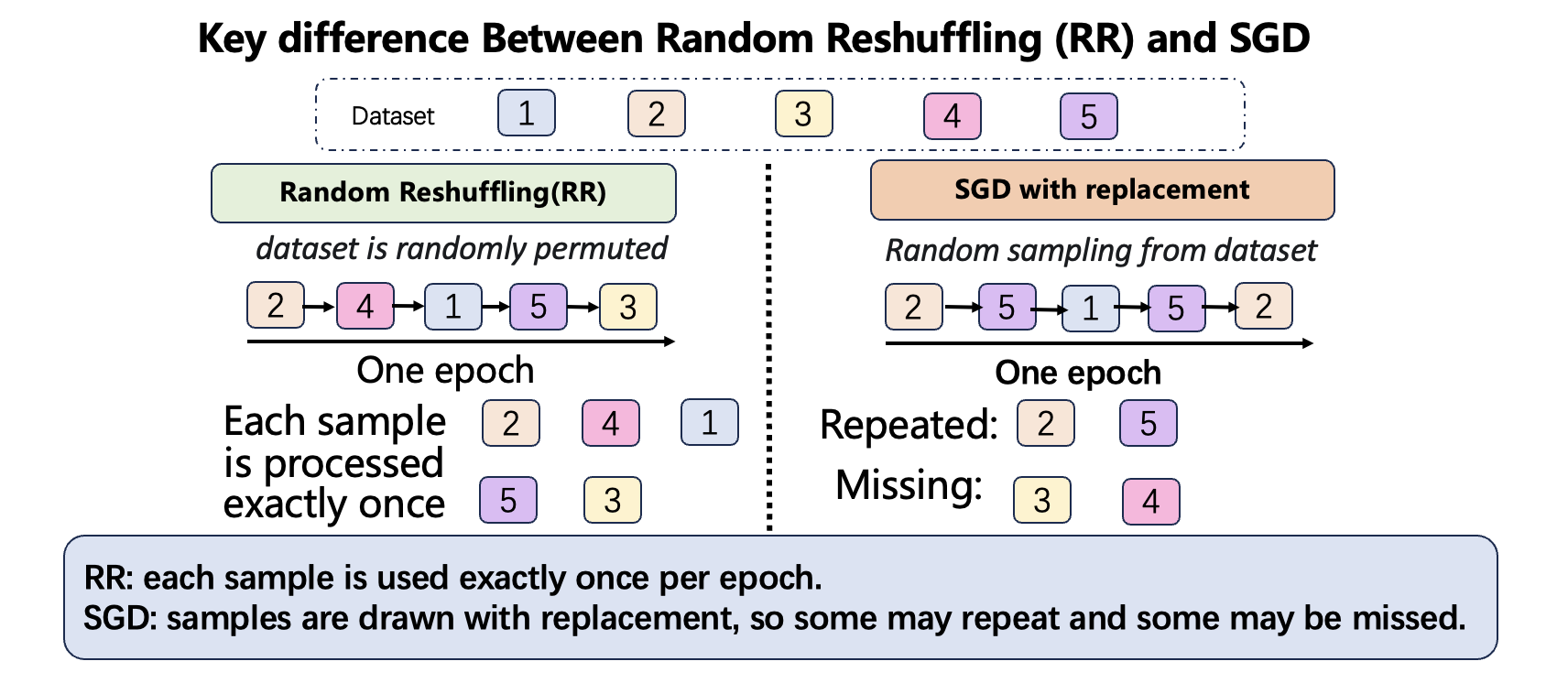}
    \caption{Key Difference between Random Reshuffling (RR) and SGD}
    \label{RR_syt}
\end{figure}
Moreover, RR is related to the full-gradient method in that both
fully exploit the finite sum structure of the sample-average objective; see,
e.g., \cite{shamir2016without,mishchenko2020random}. However, a full-gradient
step evaluates all $m$ component gradients at the same iterate, whereas RR
visits these components sequentially within one epoch. Hence, RR may be viewed
as a lower-cost incremental surrogate of a full-gradient pass. Moreover, unlike
with-replacement SGD, RR processes each component exactly once per epoch and
has been shown to enjoy improved convergence behavior after finite epochs.

\section{Random Reshuffling Based Nash Equilibrium Seeking under Partial Decision Information Case}
\label{sec:partial_information}
We consider distributed Nash equilibrium seeking under partial-decision
information, where each player maintains local estimates of the joint action and
updates them through neighbor communication. The proposed RR scheme introduces
two coupled perturbations: the within-epoch reshuffling drift and the
disagreement error caused by local estimates. Accordingly, the convergence
analysis couples the decision-error and consensus-error recursions. We first
introduce the communication model and the resulting compact dynamics. Then describe the communication model and the
resulting compact estimate dynamics.
\subsection{Communication Graph and Compact Dynamics}

Let the players communicate over an undirected graph
$\mathcal G=(\mathcal V,\mathcal E)$, where \(\mathcal V=[N]\) is the player set and
\(\mathcal E\subseteq \mathcal V\times\mathcal V\) is the edge set. For each
player \(i\), define its neighbor set as
\[
\mathcal N_i:=\{j\in[N]:(i,j)\in\mathcal E\}.
\]
Let \(a_{ij}\ge0\) be the edge weight associated with the
communication link \((i,j)\), and let \(L_{\mathcal G}\in\mathbb R^{N\times N}\)
be the corresponding weighted graph Laplacian.

\begin{assumption}
\label{ass:communication}
The communication graph \(\mathcal G=(\mathcal V,\mathcal E)\) is undirected
and connected. The edge weights \(a_{ij}\) are compatible with
\(\mathcal G\) and symmetric, that is,
\(a_{ij}>0\) only if \(j\in\mathcal N_i\), and
\(a_{ij}=a_{ji}\). 
\end{assumption}
\begin{algorithm}[h]
\caption{RR with partial-decision information}
\label{partialinformation}
\begin{algorithmic}[1]
\Require Initial action \(x_{i,0}\in\Omega_i\), initial estimate
\(\mathbf y_{i,0}\), stepsizes \(\{\alpha_k\}\), consensus stepsize
\(\{w_k\}\), adjacency weight \(a_{ij}\ge 0\)
\For{epoch \(k=0,1,2,\ldots,K-1\)}
    \State Player \(i\) independently samples a permutation
    \[
    \{\pi_0^i,\pi_1^i,\ldots,\pi_{m-1}^i\}
    \quad \text{of} \quad
    \{1,2,\ldots,m\}.
    \]
    \State \(x_{i,k}^0\gets x_{i,k}\), \(\mathbf y_{i,k}^0\gets \mathbf y_{i,k}\)
    \For{\(\ell=0,1,\ldots,m-1\)}
        \State Player \(i\) evaluates the local component pseudo-gradient
        \[
        \nabla_{x_i}f_i(\mathbf y_{i,k}^{\ell};\pi_\ell^i).
        \]
        \State Player \(i\) updates its action by
        \[
        x_{i,k}^{\ell+1}
        \gets
        x_{i,k}^{\ell}
        -
        \alpha_k
        \nabla_{x_i}f_i(\mathbf y_{i,k}^{\ell};\pi_\ell^i).
        \]
        \For{\(j=1,\ldots,N\)}
            \State  Player \(i\)'s local estimate of player
\(j\)'s action at the \(\ell+1\)-th inner step of epoch \(k\) by 
\begin{align*}
            y_{ij,k}^{\ell+1}
            \gets
            &y_{ij,k}^{\ell}\\
            -&w_k
            \left(
            \sum_{s=1}^{N}
            a_{is}
            \left(
            y_{ij,k}^{\ell}-y_{sj,k}^{\ell}
            \right)
            +
            a_{ij}
            \left(
            y_{ij,k}^{\ell}-x_{j,k}^{\ell}
            \right)
            \right).
            \end{align*}
        \EndFor
    \EndFor
    \State \(x_{i,k+1}\gets \mathbb{P}_{\Omega_i}[x_{i,k}^{m}\)], \(\mathbf y_{i,k+1}\gets \mathbf y_{i,k}^{m}\)
\EndFor
\State \Return \(x_{i,K}\)
\end{algorithmic}
\end{algorithm}
In the partial-decision-information setting, each player \(i\) maintains a local
estimate $\mathbf y_{i,k}^{\ell}:=
\operatorname{col}
(y_{i1,k}^{\ell},\ldots,y_{iN,k}^{\ell})
\in\mathbb R^N$, where \(y_{ij,k}^{\ell}\) denotes player \(i\)'s estimate of player \(j\)'s action
at the \(\ell\)-th inner step of epoch \(k\). The  estimate vector is
$
\mathbf y_k^\ell
:=
\operatorname{col}
(\mathbf y_{1,k}^{\ell},\ldots,\mathbf y_{N,k}^{\ell})
\in\mathbb R^{N^2}.
$
The true joint action is $\mathbf x_k^\ell
:=
\operatorname{col}(x_{1,k}^\ell,\ldots,x_{N,k}^\ell)
\in\mathbb R^N$.
For analysis, define the  disagreement at the $\ell-$th inner step of epoch $k$:
\begin{align}\label{consensus}
    \bar{\mathbf y}_k^\ell
:=
\mathbf y_k^\ell-\mathbf 1_N\otimes \mathbf x_k^\ell.
\end{align}
The vector \(\mathbf 1_N\otimes \mathbf x_k^\ell\) is only an analysis reference
for measuring disagreement; it is not used by any player in the implementation.
\
For each epoch $k$, $x_{i,k}^\ell$ denotes the intermediate action of player $i$
after $\ell$ reshuffled component updates, with $x_{i,k}^0:=x_{i,k}$. Then the compact form of the epoch-end projection is $\mathbf x_{k+1}=\mathbb P_{\Omega}[\mathbf x_k^m]$.
For a given reshuffled index tuple $\pi_\ell:=(\pi_\ell^1,\ldots,\pi_\ell^N)$, define the component pseudo-gradient as
\[
\nabla F_{\pi_\ell}(x)
:=
\operatorname{col}
\left(
\nabla_{x_1}f_1(x;\pi_\ell^1),
\ldots,
\nabla_{x_N}f_N(x;\pi_\ell^N)
\right).
\]
Since each player visits every local component exactly once within an epoch, we have
\[
\sum_{\ell=0}^{m-1}\nabla F_{\pi_\ell}(x)
=
m\nabla F(x),
\]
where $\nabla F(x)$ denotes the averaged pseudo-gradient of the game.

The update in Algorithm~\ref{partialinformation} is distributed. Player \(i\)
uses only its own action \(x_{i,k}^{\ell+1}\) and the estimate vectors received
from its neighbors. The compact expressions below are used only for convergence
analysis after stacking all local variables.

Define the  local-estimate pseudo-gradient by
\[
\widehat G_{\pi_\ell}(\mathbf y_k^\ell)
:=
\operatorname{col}
\left(
\nabla_{x_1}f_1(\mathbf y_{1,k}^{\ell};\pi_\ell^1),
\ldots,
\nabla_{x_N}f_N(\mathbf y_{N,k}^{\ell};\pi_\ell^N)
\right).
\]
For a true action profile \(\mathbf x\in\mathbb R^N\), define
\begin{align*}
  \widehat G_{\pi_\ell}(\mathbf 1_N\otimes \mathbf x)&:=\nabla F_{\pi_\ell}(\mathbf x)\\
                                                     &= \operatorname{col}
\left(
\nabla_{x_1}f_1(\mathbf x;\pi_\ell^1),
\ldots,
\nabla_{x_N}f_N(\mathbf x;\pi_\ell^N)
\right).
\end{align*}
Then the  action recursion is
\begin{align}
\mathbf x_k^{\ell+1}
&=
\mathbf x_k^\ell
-
\alpha_k\widehat G_{\pi_\ell}(\mathbf y_k^\ell),
\qquad \ell=0,\ldots,m-1 .
\label{eq:partial_x_inner}
\end{align}
Equivalently, using \eqref{consensus}, we have
\begin{align}
\label{compact_formyk}
\mathbf y_k^{\ell+1}
=
\mathbf y_k^\ell-w_kH\bar{\mathbf y}_k^\ell.
\end{align}
The local estimate recursion induces the following compact disagreement dynamics:
\begin{align}
\bar{\mathbf y}_k^{\ell+1}
&=
(I_{N^2}-w_kH)\bar{\mathbf y}_k^\ell
+
\alpha_k\mathbf 1_N\otimes
\widehat G_{\pi_\ell}(\mathbf y_k^\ell).
\label{eq:partial_disagreement_inner}
\end{align}
Here, $H:=L_{\mathcal G}\otimes I_N+\Delta
\in\mathbb R^{N^2\times N^2}
$, where \(L_{\mathcal G}\) is the weighted graph Laplacian induced by the edge
weights \(a_{ij}\), and
$
\Delta
:=
\operatorname{diag}
(a_{11},a_{12},\ldots,a_{1N},a_{21},\ldots,a_{NN})
\in\mathbb R^{N^2\times N^2}$.
We assume
that the connected graph make \(H\) positive
definite. consensus stepsize \(w_k>0\) is chosen such that $q:=\|I_{N^2}-w_kH\|<1$.\par
For example, in the Euclidean norm this holds whenever \(H\) is symmetric
positive definite and $0<w_k<\frac{2}{\lambda_{\max}(H)}$. \par
In the partial-decision-information case, the component pseudo-gradient is
evaluated at local estimates rather than at the true joint action. Therefore,
the regularity required for the analysis must hold not only near the feasible
action set \(\Omega\), but also near the consensus manifold
\(\{\mathbf 1_N\otimes x:x\in\Omega_R\}\) in the  estimate space. The
set \(\mathcal D_{R,Y}\) below is a local neighborhood of this consensus
manifold. The radius \(Y\) is fixed a priori, and the subsequent boundedness
result shows that, under a sufficiently small stepsize, the estimate trajectory
remains inside this neighborhood.
\subsection{Local Domain For The Partial-Information Analysis}

Since projection is performed only at the end of each epoch, the inner
decision iterates may temporarily leave the feasible set. Moreover, under
partial-decision information, component pseudo-gradients are evaluated at
local estimates rather than at the true joint action. We therefore introduce
the following local domains only for the analysis.

For fixed constants \(R>0\) and \(Y>0\), define
\[
\Omega_R:=\{x\in\mathbb R^N:\operatorname{dist}(x,\Omega)\le R\},
\]
and
\[
\mathcal C_R:=\{\mathbf 1_N\otimes x:x\in\Omega_R\},
\]
\[
\mathcal D_{R,Y}:=
\left\{
\mathbf y\in\mathbb R^{N^2}:
\operatorname{dist}(\mathbf y,\mathcal C_R)\le Y
\right\}.
\]
Here, \(\Omega_R\) is an enlarged neighborhood of the feasible set, and
\(\mathcal D_{R,Y}\) is a \(Y\)-neighborhood of the consensus manifold over
\(\Omega_R\). Since \(\Omega\) is compact, both \(\Omega_R\) and
\(\mathcal D_{R,Y}\) are compact.

\begin{lemma}
\label{cor:partial_inner_boundedness}
Fix \(R>0\) and \(Y>0\). Assume that, on the local domains defined above,
the following local regularity conditions hold. For every component index tuple
\(\pi_\ell\), the local-estimate pseudo-gradient \(\widehat G_{\pi_\ell}\) is
well defined and \(L_y\)-Lipschitz continuous on \(\mathcal D_{R,Y}\), and the
true-action component pseudo-gradient \(\nabla F_{\pi_\ell}\) is
\(L_x\)-Lipschitz continuous on \(\Omega_R\). Moreover, the consistency relation
\[
\widehat G_{\pi_\ell}(\mathbf 1_N\otimes x)
=
\nabla F_{\pi_\ell}(x),
\qquad
\forall x\in\Omega_R,
\]
holds. Let \(L:=\max\{L_x,L_y\}\).

Since \(\Omega_R\) is compact and the number of component index tuples is
finite, define
\[
G_R:=
\max_{\pi_\ell}
\sup_{x\in\Omega_R}
\|\nabla F_{\pi_\ell}(x)\|
<\infty,
\qquad
M_Y:=G_R+LY .
\]
Assume further that the communication stepsize satisfies $\|I-w_kH\|\le q<1 \;\text{for any} \;  k$. Let
\[
\bar\alpha_Y
:=
\min
\left\{
\frac{R}{mM_Y},
\frac{(1-q)Y}{\sqrt N M_Y},
\frac{(1-q^m)Y}
{
\sqrt N M_Y
\left(
\frac{1-q^m}{1-q}+m
\right)
}
\right\}.
\]
If \(0<\alpha_k\le \bar\alpha_Y\) and
\(\|\bar{\mathbf y}_0\|\le Y\), then for all \(k\ge0\) and
\[
\mathbf x_k^\ell\in\Omega_R,  
\qquad
\|\bar{\mathbf y}_k^\ell\|\le Y, \qquad \ell=0,\ldots,m.
\]
Moreover,
\[
\|\mathbf x_k^\ell-\mathbf x_k\|
\le
\alpha_k\ell M_Y,
\qquad
\ell=0,\ldots,m.
\]
\end{lemma}
\begin{proof}
See proof in Appendix~\ref{collary1}
\end{proof}

lemma~\ref{cor:partial_inner_boundedness} guarantees that all gradients
evaluated along the partial-information RR inner loop are uniformly bounded. Based
on this property, we next derive an epoch-level recursion for the consensus error.

\begin{lemma} \label{lem:partial_consensus} Suppose Assumptions~1--\ref{ass:communication} hold, and the conditions of Lemma~\ref{cor:partial_inner_boundedness} are satisfied. Assume that the consensus stepsize $\{w_k\}$ is chosen such that 
\begin{equation*} 
\|I-w_kH\|\le q<1,\qquad \forall k . 
\end{equation*} 
Let 
\begin{equation*} 
s:=\frac{1+q^2}{2},\quad c_Y:=\frac{1+q^2}{1-q^2}NM_Y^2,\quad S_m:=\frac{1-s^m}{1-s}. 
\end{equation*} Then, for all \(k\), 
\begin{equation} 
\sum_{\ell=0}^{m-1}\|\bar{\mathbf y}_k^\ell\|^2 \le S_m\|\bar{\mathbf y}_k\|^2 + c_Y\frac{m-S_m}{1-s}\alpha_k^2 . \label{eq:inner_consensus_sum} 
\end{equation} 
Moreover, with 
\begin{equation*} 
r_Y:=\frac{1+s^m}{2},\qquad D_Y:= \frac{1+s^m}{2s^m}c_YS_m + \frac{1+s^m}{1-s^m}Nm^2M_Y^2, 
\end{equation*} 
the epoch-level consensus error satisfies 
\begin{equation} 
B_{k+1}\le r_YB_k+D_Y\alpha_k^2, \qquad B_k:=\mathbb E\|\bar{\mathbf y}_k\|^2 . \label{eq:epoch_consensus_recursion} 
\end{equation} 
Consequently, if \(\alpha_k\equiv\alpha\), then 
\begin{equation} 
B_k\le r_Y^kB_0+\frac{D_Y}{1-r_Y}\alpha^2 . \label{eq:constant_consensus_bound} 
\end{equation} 
\end{lemma}
\begin{proof}
See proof in Appendix~\ref{Lemma2}.
\end{proof}
Lemma~\ref{lem:partial_consensus} is the key estimate that connects the
communication dynamics with the RR inner loop. It shows that the local estimates
reach an $O(\alpha^2)$ mean-square consensus neighborhood under constant
stepsizes. We now combine this consensus recursion with the decision-error
recursion.

\subsection{Constant-Stepsize Convergence}

\begin{theorem}
\label{thm:partial_constant}
Suppose Assumptions~1--\ref{ass:communication} hold.
Let the local regularity and boundedness conditions in
Lemma~\ref{cor:partial_inner_boundedness} and the conditions of
Lemma~\ref{lem:partial_consensus} be satisfied. Consider
Algorithm~\ref{partialinformation} with constant stepsize $\alpha_k\equiv\alpha$.
Assume that $0<\alpha\le \bar\alpha_{\mathrm p}$, where
$
\bar\alpha_{\mathrm p}
:=
\min
\left\{
\bar\alpha_Y,
\frac{\mu}{mL^2},
\frac{1}{2m\mu}
\right\}$.
Define $\rho_X:=1-\frac{m\mu\alpha}{2}$ and $\bar\rho:=\max\{\rho_X,r_Y\}$. Then \(0<\rho_X<1\) and \(0<\bar\rho<1\). Moreover, there exist positive
constants \(C_1\) and \(C_2\), depending only on the problem and network
parameters but independent of \(K\) and \(\alpha\), such that
\[
\mathbb E\|\mathbf x_K-\mathbf x_\star\|^2
\le
\rho_X^K
\mathbb E\|\mathbf x_0-\mathbf x_\star\|^2
+
C_1\alpha^2
+
C_2\alpha K\bar\rho^{K-1}.
\]
Consequently,
\[
\limsup_{K\to\infty}
\mathbb E\|\mathbf x_K-\mathbf x_\star\|^2
\le
C_1\alpha^2.
\]
In addition, the consensus error satisfies
\[
\mathbb E\|\bar{\mathbf y}_K\|^2
\le
r_Y^K\mathbb E\|\bar{\mathbf y}_0\|^2
+
\frac{D_Y}{1-r_Y}\alpha^2.
\]
Thus, the decision and consensus errors converge to \(O(\alpha^2)\)
mean-square neighborhoods.
\begin{proof}
See proof in Appendix~\ref{Theorem1}.
\end{proof}
\end{theorem}
\subsection{Diminishing-Stepsize Convergence}

We next show that exact convergence can be recovered by using a diminishing
stepsize. The proof combines the decision-error recursion with the
consensus-error recursion and then applies a deterministic
Robbins--Siegmund-type argument to the mean-square decision error.

\begin{lemma}
\label{lem:det_rs}
Let \(\{v_k\}\), \(\{b_k\}\), and \(\{c_k\}\) be nonnegative deterministic
sequences satisfying
\[
v_{k+1}\le v_k-b_k+c_k,
\qquad
\forall k\ge0.
\]
If $\sum_{k=0}^{\infty}c_k<\infty$, then \(\{v_k\}\) converges to a finite limit and $\sum_{k=0}^{\infty}b_k<\infty$.

\end{lemma}

\begin{theorem}
\label{thm:partial_diminishing}
Suppose Assumptions~1--\ref{ass:communication} hold.
Let the local regularity and boundedness conditions in
Lemma~\ref{cor:partial_inner_boundedness} and the conditions of
Lemma~\ref{lem:partial_consensus} be satisfied. Consider
Algorithm~\ref{partialinformation} with stepsizes \(\{\alpha_k\}\) satisfying
\[
0<\alpha_k\le \bar\alpha_{\mathrm d},
\qquad
\sum_{k=0}^{\infty}\alpha_k=\infty,
\qquad
\sum_{k=0}^{\infty}\alpha_k^2<\infty,
\]
where
\[
\bar\alpha_{\mathrm d}
:=
\min
\left\{
\bar\alpha_Y,
\frac{\mu}{mL^2},
\frac{1}{2m\mu}
\right\}.
\]
Then
\[
\lim_{k\to\infty}
\mathbb E\|\mathbf x_k-\mathbf x_\star\|^2=0,
\qquad
\lim_{k\to\infty}
\mathbb E\|\bar{\mathbf y}_k\|^2=0.
\]
In particular, \(\mathbf x_k\) converges to \(\mathbf x_\star\) in mean square.
\end{theorem}

\begin{proof}
See Appendix~\ref{Theorem2}.
\end{proof}

\section{Simulation}

We evaluate the proposed RR based stochastic pseudo-gradient method
on a scenario-based distributed Nash game under partial-decision information.
The compared methods are the proposed player-wise RR method, with-replacement
SGD, deterministic cyclic ordering, fixed reshuffling, common-permutation RR,
and mini-batch SGD.

All methods use the same communication graph, estimate-consensus recursion,
projection rule, initialization, and stopping horizon. They differ only in the
sampling rule used to select local component pseudo-gradients. In
with-replacement SGD, each player independently samples a component index with
replacement at each inner iteration. In the proposed RR method, each player
independently draws a fresh random permutation at every epoch and visits every
local component exactly once. The cyclic method uses the fixed order
\(1,\ldots,m\) at every epoch, while fixed reshuffling uses one random
permutation fixed across all epochs. Common-permutation RR uses one fresh
permutation shared by all players.

We report two metrics. The first is the squared decision error
\[
\|\mathbf x_k-\mathbf x^\star\|^2,
\]
where \(\mathbf x^\star\) is a high-accuracy reference Nash equilibrium. The
second is the estimate-consensus error
\[
\|\bar{\mathbf y}_k\|^2
=
\left\|
\mathbf y_k-\mathbf 1_N\otimes \mathbf x_k
\right\|^2,
\]
which measures the disagreement between the players' local estimates and the
current joint action profile under partial-decision information.

The curves are averaged over \(30\) independent runs, and shaded regions
represent \(95\%\) confidence intervals. Terminal performance is reported as
the average over the last \(20\%\) of the recorded iterates.

\subsubsection{Scenario-Based EV Charging Game}

We consider a multi-period EV charging game with \(N=20\) regional charging
aggregators and \(T=24\) time slots. Player \(i\) chooses a charging profile
\(x_i\in\mathbb R^{T}\) subject to
\[
\Omega_i
=
\left\{
x_i\in\mathbb R^T:
0\le x_{i,t}\le \bar p_i,\;
\sum_{t=1}^{T}x_{i,t}=E_i
\right\},
\]
where \(E_i\in[10,20]\) is the required energy and
\(\bar p_i\in[4.5,7.5]\) is the charging-power limit. The finite-sum objective
of player \(i\) is $ f_i(\mathbf x) = \frac{1}{m}\sum_{\ell=1}^{m} f_i(\mathbf x;\ell)$
with $m=128$.

For scenario \(\ell\), the component cost is
\begin{align*}
   &f_i(\mathbf x;\ell)\\
=&
\frac{a_i}{2}\|x_i-p_i\|^2
+
\lambda_\ell^\top x_i
+
\frac{\rho_\ell}{2}
\left\|
\sum_{j=1}^{N}x_j+b_\ell-c_\ell
\right\|^2
+
\frac{\tau}{2}\|x_i\|^2 . 
\end{align*}
Here, \(p_i\) is the preferred charging profile,
\(a_i\in[0.4,0.8]\), \(\tau=0.02\), \(\lambda_\ell\) is the scenario price,
\(b_\ell\) is the non-EV base load, \(c_\ell\) is the feeder capacity, and
\(\rho_\ell\) is the congestion coefficient generated around \(0.015\).

The component pseudo-gradient is $\nabla_{x_i} f_i(\mathbf x;\ell)
=
a_i(x_i-p_i)
+
\lambda_\ell
+
\rho_\ell
\left(
\sum_{j=1}^{N}x_j+b_\ell-c_\ell
\right)
+
\tau x_i $.
Under partial-decision information, player \(i\) evaluates this pseudo-gradient
using its local estimate of the joint action, rather than the true global
action profile. Players exchange estimate vectors only with their communication
neighbors.\par

The communication graph is a ring graph with \(20\) edges and degree \(2\).
Metropolis weights are used, yielding the disagreement contraction factor
\(q=0.9674\). The reference Nash equilibrium is computed offline by a
high-accuracy projected pseudo-gradient method, with numerical residual
\(1.31\times 10^{-20}\). All algorithms are run for \(1500\) epochs.\par
To isolate the effect of the reshuffling rule, we compare the proposed
Algorithm~\ref{partialinformation} with several sampling-rule variants under
the same partial-information communication and projection framework. The
proposed player-wise RR is exactly Algorithm~\ref{partialinformation}: each
player independently generates its own random permutation at every epoch.
Common RR keeps the same update structure but forces all players to use a
common epoch-wise permutation. Fixed RR also follows Algorithm~\ref{partialinformation}
but reuses a fixed random permutation across epochs. The cyclic incremental
method replaces the random permutation by the deterministic order
\(1,\ldots,m\). The with-replacement SGD and mini-batch SGD baselines replace
the without-replacement pass by independent sampling with replacement. All
methods therefore share the same estimate-consensus recursion, epoch-end
projection, initialization, stepsize, and communication graph; they differ only
in how local component pseudo-gradients are selected within each epoch.\par
\begin{figure}[t]
    \centering
    \includegraphics[width=1.0\linewidth]{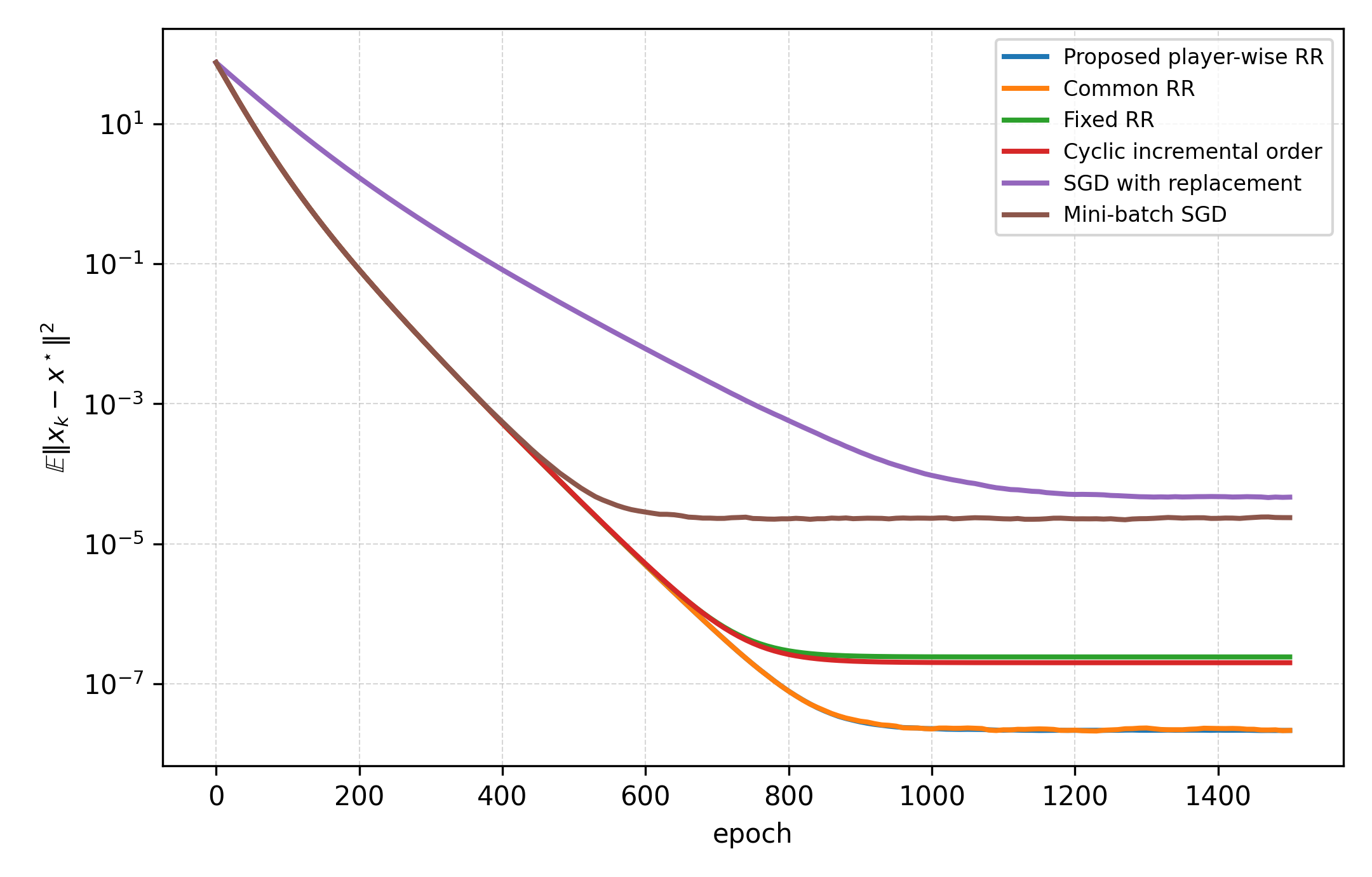}
    \includegraphics[width=1.0\linewidth]{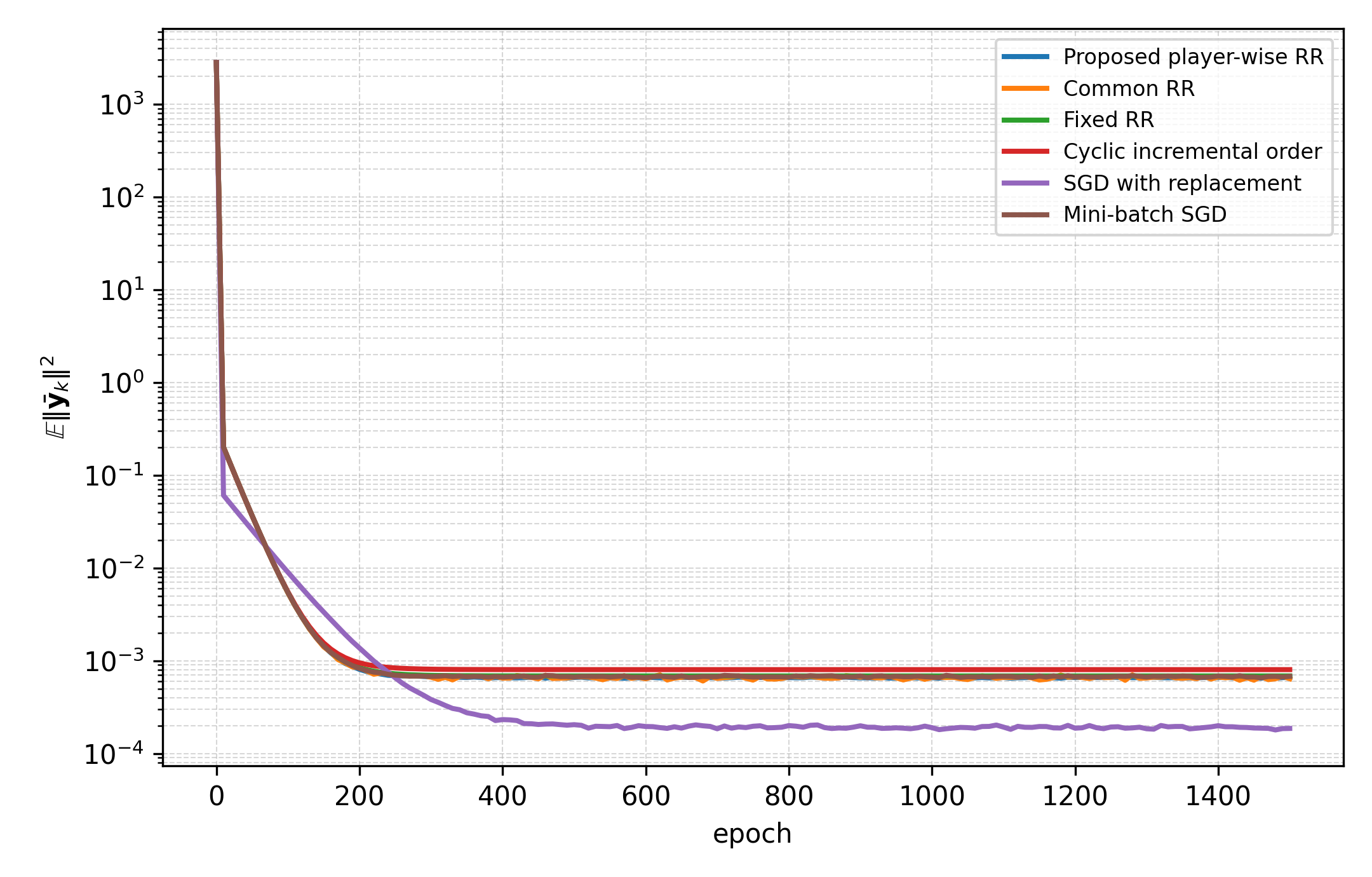}
    \caption{Convergence behavior on the scenario-based EV charging game.
    Above: squared decision error for the core methods. Below:
    estimate-consensus error under partial-decision information. The proposed
    RR method attains substantially smaller decision error than
    with-replacement SGD, while the local-estimate disagreement remains
    bounded.}
    \label{fig:ev_game}
\end{figure}

Fig.~\ref{fig:ev_game} shows that the proposed player-wise RR method converges
to a substantially smaller decision-error neighborhood than with-replacement
SGD. The cyclic-order baseline is included as a deterministic
without-replacement method; its performance shows that the gain of RR is not
only due to using each component once per epoch, but also to drawing a fresh
random permutation.

The consensus-error curve verifies the stability of the distributed
partial-information estimate recursion. This metric measures disagreement
among local estimates and is not, by itself, an equilibrium-accuracy metric.
Thus, a smaller consensus error does not necessarily imply a smaller decision
error.

Over \(30\) independent runs, the proposed RR method achieves a tail squared
decision error of \(2.165\times10^{-8}\), compared with
\(4.757\times10^{-5}\) for with-replacement SGD. The tail consensus errors of RR and
with-replacement SGD are \(6.588\times10^{-4}\) and
\(1.905\times10^{-4}\), respectively, indicating that both methods keep the
local-estimate disagreement bounded under the same communication protocol.

\section{Conclusion}
This paper investigated random-reshuffling-based Nash equilibrium seeking for noncooperative games. We considered partial-decision-information
case. By using a VI-based epoch-wise drift decomposition, we characterized the
RR-induced inner-trajectory error and its coupling with local-estimate
disagreement. Under constant stepsizes, the proposed methods converge linearly
to an \(O(\alpha^2)\) mean-square neighborhood of the Nash equilibrium, while
under diminishing stepsizes exact mean-square convergence is obtained. Numerical
results show that RR consistently improves terminal accuracy over with-replacement
SGD under comparable oracle complexity. Future work will focus on sharper
permutation-dependent bounds, relaxing strong monotonicity, and extending the
framework to broader classes of stochastic games.
\appendix
\section*{Appendix}

\section{Proof of Lemma 1}\label{collary1}
\begin{proof}
Let \(g_k^\ell:=\widehat G_{\pi_\ell}(\mathbf y_k^\ell)\). We prove the result
by induction over epochs:\par
Since the algorithm performs projection at the end of
each epoch, \(\mathbf x_k\in\Omega\) for all \(k\ge0\), and hence
\(\mathbf x_k^0=\mathbf x_k\in\Omega\subseteq\Omega_R\). Suppose that \(\|\bar{\mathbf y}_k\|\le Y\) at the
beginning of epoch \(k\). We show that all inner
iterates in epoch \(k\) remain in the desired bounded region and that the same
condition also holds at the beginning of epoch \(k+1\).

For \(\ell=0\), \(\mathbf x_k^0=\mathbf x_k\in\Omega_R\) and
\(\|\bar{\mathbf y}_k^0\|=\|\bar{\mathbf y}_k\|\le Y\), so
\(\mathbf y_k^0\in\mathcal D_{R,Y}\). More generally, if
\(\mathbf x_k^\ell\in\Omega_R\) and \(\|\bar{\mathbf y}_k^\ell\|\le Y\), then
\(\mathbf y_k^\ell\in\mathcal D_{R,Y}\). By the consistency condition in
Lemma~\ref{cor:partial_inner_boundedness} and
\(\widehat G_{\pi_\ell}(\mathbf 1_N\otimes\mathbf x_k^\ell)
=\nabla F_{\pi_\ell}(\mathbf x_k^\ell)\) we have
\begin{align*}
\|g_k^\ell\|
&=\|\widehat G_{\pi_\ell}(\mathbf y_k^\ell)\|\\
&\le \|\widehat G_{\pi_\ell}(\mathbf 1_N\otimes\mathbf x_k^\ell)\|
+\|\widehat G_{\pi_\ell}(\mathbf y_k^\ell)
-\widehat G_{\pi_\ell}(\mathbf 1_N\otimes\mathbf x_k^\ell)\|\\
&\le G_R+L\|\mathbf y_k^\ell-\mathbf 1_N\otimes\mathbf x_k^\ell\|
\le G_R+LY=M_Y.
\end{align*}
Thus,
\(\|\mathbf x_k^{\ell+1}-\mathbf x_k^\ell\|
=\alpha_k\|g_k^\ell\|\le\alpha_kM_Y\), and, for
\(\ell=0,1,\ldots,m\),
\begin{align*}
\|\mathbf x_k^\ell-\mathbf x_k\|
&=\left\|\sum_{p=0}^{\ell-1}
(\mathbf x_k^{p+1}-\mathbf x_k^p)\right\|
\le\sum_{p=0}^{\ell-1}\|\mathbf x_k^{p+1}-\mathbf x_k^p\|\\
&\le\alpha_k\ell M_Y.
\end{align*}
Since \(\ell\le m\) and \(\mathbf x_k\in\Omega\),
\[
\operatorname{dist}(\mathbf x_k^\ell,\Omega)
\le\|\mathbf x_k^\ell-\mathbf x_k\|
\le\alpha_kmM_Y.
\]
If \(\alpha_k\le R/(mM_Y)\), then
\(\operatorname{dist}(\mathbf x_k^\ell,\Omega)\le R\), and hence
\(\mathbf x_k^\ell\in\Omega_R\).

Next, we prove that the consensus error remains bounded by \(Y\) during the
inner loop. From
\(\mathbf y_k^{\ell+1}=\mathbf y_k^\ell-w_kH\bar{\mathbf y}_k^\ell\) and
\(\mathbf x_k^{\ell+1}=\mathbf x_k^\ell-\alpha_kg_k^\ell\), we obtain
\begin{align*}
\bar{\mathbf y}_k^{\ell+1}
&=\mathbf y_k^{\ell+1}-\mathbf 1_N\otimes\mathbf x_k^{\ell+1}\\
&=(I-w_kH)\bar{\mathbf y}_k^\ell
+\alpha_k\mathbf 1_N\otimes g_k^\ell.
\end{align*}
Taking norms gives
\[
\|\bar{\mathbf y}_k^{\ell+1}\|
\le q\|\bar{\mathbf y}_k^\ell\|+\alpha_k\sqrt N M_Y.
\]
Unrolling this recursion and using
\(\|\bar{\mathbf y}_k^0\|=\|\bar{\mathbf y}_k\|\le Y\) yields
\[
\|\bar{\mathbf y}_k^\ell\|
\le q^\ell Y+\alpha_k\sqrt N M_Y\frac{1-q^\ell}{1-q}.
\]
If \(\alpha_k\le(1-q)Y/(\sqrt N M_Y)\), then
\[
\alpha_k\sqrt N M_Y\frac{1-q^\ell}{1-q}\le Y(1-q^\ell),
\]
and therefore \(\|\bar{\mathbf y}_k^\ell\|\le Y\) for every
\(\ell=0,\ldots,m\).

It remains to show that the next epoch also starts with a consensus error not
larger than \(Y\). By the epoch-end update,
\(\mathbf y_{k+1}=\mathbf y_k^m\) and
\(\mathbf x_{k+1}=\mathbb P_\Omega[\mathbf x_k^m]\), so
\begin{align*}
\bar{\mathbf y}_{k+1}
&=\bar{\mathbf y}_k^m+
\mathbf 1_N\otimes
\bigl(\mathbf x_k^m-\mathbb P_\Omega[\mathbf x_k^m]\bigr).
\end{align*}
Thus,
\[
\|\bar{\mathbf y}_{k+1}\|
\le\|\bar{\mathbf y}_k^m\|
+\sqrt N\|\mathbf x_k^m-\mathbb P_\Omega[\mathbf x_k^m]\|.
\]
The inner-loop estimate with \(\ell=m\) and the inner-trajectory bound give
\begin{align*}
\|\bar{\mathbf y}_k^m\|
&\le q^mY+\alpha_k\sqrt N M_Y\frac{1-q^m}{1-q},
\end{align*}
and
\begin{align*}
\|\mathbf x_k^m-\mathbb P_\Omega[\mathbf x_k^m]\|
&=\operatorname{dist}(\mathbf x_k^m,\Omega)
\le\|\mathbf x_k^m-\mathbf x_k\|
\le\alpha_kmM_Y.
\end{align*}
Consequently,
\[
\|\bar{\mathbf y}_{k+1}\|
\le q^mY+\alpha_k\sqrt N M_Y
\left(\frac{1-q^m}{1-q}+m\right).
\]
If $\alpha_k\le
\frac{(1-q^m)Y}
{\sqrt N M_Y\left(\frac{1-q^m}{1-q}+m\right)}$,
then \(\|\bar{\mathbf y}_{k+1}\|\le q^mY+(1-q^m)Y=Y\).

Therefore, if \(0<\alpha_k\le\bar\alpha_Y\), where
\[
\bar\alpha_Y:=
\min\left\{
\frac{R}{mM_Y},
\frac{(1-q)Y}{\sqrt N M_Y},
\frac{(1-q^m)Y}
{\sqrt N M_Y\left(\frac{1-q^m}{1-q}+m\right)}
\right\},
\]
then
\(\|\bar{\mathbf y}_k\|\le Y\Rightarrow
\|\bar{\mathbf y}_{k+1}\|\le Y\).
Since \(\|\bar{\mathbf y}_0\|\le Y\), induction over \(k\) gives
\(\|\bar{\mathbf y}_k\|\le Y\) for all \(k\ge0\). The preceding estimates also
imply, for all \(k\ge0\) and \(\ell=0,\ldots,m\),
\[
\mathbf x_k^\ell\in\Omega_R,\qquad
\|\bar{\mathbf y}_k^\ell\|\le Y,\qquad
\|\mathbf x_k^\ell-\mathbf x_k\|\le\alpha_k\ell M_Y.
\]
This completes the proof.
\end{proof}

\section{Proof of Lemma 2}\label{Lemma2}
\begin{proof}
Define
\(\bar{\mathbf y}_k^\ell:=
\mathbf y_k^\ell-\mathbf 1_N\otimes\mathbf x_k^\ell\).
By \eqref{compact_formyk},
\(\mathbf y_k^{\ell+1}
=\mathbf y_k^\ell-w_kH\bar{\mathbf y}_k^\ell\), where \(w_k\) satisfies
\(\|I-w_kH\|\le q <1\). The inner update gives
\[
\bar{\mathbf y}_k^{\ell+1}
=(I-w_kH)\bar{\mathbf y}_k^\ell
+\alpha_k\mathbf 1_N\otimes
\widehat G_{\pi_\ell}(\mathbf y_k^\ell).
\]
Taking squared norms and using Young's inequality yields
\begin{align*}
\|\bar{\mathbf y}_k^{\ell+1}\|^2
&\le(1+\theta)q^2\|\bar{\mathbf y}_k^\ell\|^2
+\left(1+\frac1\theta\right)\alpha_k^2
\|\mathbf 1_N\otimes\widehat G_{\pi_\ell}(\mathbf y_k^\ell)\|^2\\
&\le(1+\theta)q^2\|\bar{\mathbf y}_k^\ell\|^2
+\left(1+\frac1\theta\right)\alpha_k^2NM_Y^2.
\end{align*}
Set \(s:=(1+\theta)q^2\) and
\(c_Y:=(1+1/\theta)NM_Y^2\). Choosing
\(\theta=(1-q^2)/(2q^2)\) gives \(s<1\). Unrolling across \(\ell\) gives
\begin{equation}\label{10}
\|\bar{\mathbf y}_k^\ell\|^2
\le s^\ell\|\bar{\mathbf y}_k^0\|^2
+c_Y\alpha_k^2\sum_{r=0}^{\ell-1}s^r.
\end{equation}

At the epoch boundary,
\begin{align*}
\bar{\mathbf y}_{k+1}
&=\mathbf y_k^m-\mathbf 1_N\otimes\mathbb P_\Omega[\mathbf x_k^m]\\
&=\bar{\mathbf y}_k^m+
\mathbf 1_N\otimes
\bigl(\mathbf x_k^m-\mathbb P_\Omega[\mathbf x_k^m]\bigr).
\end{align*}
Taking squared norms, using Young's inequality, the projection-distance bound,
and \eqref{10}, we obtain
\begin{align*}
\|\bar{\mathbf y}_{k+1}\|^2
&\le(1+\gamma)\|\bar{\mathbf y}_k^m\|^2
+\left(1+\frac1\gamma\right)N
\|\mathbf x_k^m-\mathbb P_\Omega[\mathbf x_k^m]\|^2\\
&\le(1+\gamma)s^m\|\bar{\mathbf y}_k^0\|^2
+(1+\gamma)c_Y\alpha_k^2\sum_{r=0}^{m-1}s^r\\
&+\left(1+\frac1\gamma\right)N\alpha_k^2m^2M_Y^2\\
&=r_Y\|\bar{\mathbf y}_k\|^2+D_Y\alpha_k^2,
\end{align*}
where
\[
r_Y:=(1+\gamma)s^m,\qquad
\]
\[
D_Y:=(1+\gamma)c_Y\sum_{r=0}^{m-1}s^r
+\left(1+\frac1\gamma\right)Nm^2M_Y^2.
\]
Choosing \(\gamma=(1-s^m)/(2s^m)\) gives
\(r_Y=(1+s^m)/2<1\). Taking expectations and defining
\(B_k:=\mathbb E\|\bar{\mathbf y}_k\|^2\), we obtain
\[
B_{k+1}\le r_YB_k+D_Y\alpha_k^2.
\]
If \(\alpha_k\equiv\alpha\), unrolling gives
\[
B_k\le r_Y^kB_0+D_Y\alpha^2\sum_{t=0}^{k-1}r_Y^t
\le r_Y^kB_0+\frac{D_Y}{1-r_Y}\alpha^2.
\]
\end{proof}

\section{Proof of Theorem 1}\label{Theorem1}
\begin{proof}
Based on \eqref{eq:partial_x_inner}, \eqref{compact_formyk} and the nonexpansiveness of projection, we have
\begin{align*}
   &\mathbb{E}[\|\mathbf{x}_{k+1}-\mathbf{x}_\star\|^2]\\
  =&\mathbb{E}[\|\mathbb{P}_{\Omega}[\mathbf{x}_{k}-\alpha_k\sum_{\ell=0}^{m-1}\widehat G_{\pi_\ell}(\mathbf y_k^\ell)]-\mathbb{P}_{\Omega}[\mathbf{x}_\star-m\alpha_k \nabla F(\mathbf{x}_\star)]\|^2]\\
\le&\mathbb{E}[\|\mathbf{x}_{k}-\mathbf{x}_\star-\alpha_k\sum_{\ell=0}^{m-1}\widehat G_{\pi_\ell}(\mathbf y_k^\ell)+m\alpha_k \nabla F(\mathbf{x}_\star)\|^2]\\
  =&\mathbb{E}[\|\mathbf{x}_{k}-\mathbf{x}_\star-m\alpha_k(\nabla F(\mathbf{x}_{k})-\nabla F(\mathbf{x}_\star))-\alpha_k (\xi_k+\delta_k)\|^2].
 \end{align*}
Here, we set  $\xi_k=\sum_{\ell=0}^{m-1} \nabla F_{\pi_\ell}(\mathbf{x}_{k}^\ell)-\nabla F_{\pi_\ell}(\mathbf{x}_{k})$ and $\delta_k= \sum_{\ell=0}^{m-1}[\widehat G_{\pi_\ell}(\mathbf y_k^\ell)-\nabla F_{\pi_\ell}(\mathbf{x}_{k}^\ell)]$. First equality holds by the variational-inequality characterization of the Nash equilibrium: $\mathbf x_\star
=
\mathbb P_\Omega[
\mathbf x_\star-m\alpha_k\nabla F(\mathbf x_\star)
]$.

Young's inequality and strong monotonicity and \(L\)-smoothness of
\(\nabla F\) yield
\begin{align*}
&\mathbb E\|\mathbf x_{k+1}-\mathbf x_\star\|^2\\
\le&(1+\eta)(1-2m\mu\alpha_k+m^2L^2\alpha_k^2)
\mathbb E\|\mathbf x_k-\mathbf x_\star\|^2\\
&\quad+\left(1+\frac1\eta\right)
\alpha_k^2\mathbb E\|\xi_k+\delta_k\|^2.
\end{align*}

By Lemma~\ref{cor:partial_inner_boundedness},
\[
\|\xi_k\|
\le\sum_{\ell=0}^{m-1}L\|\mathbf x_k^\ell-\mathbf x_k\|
\le\alpha_kLM_Y\frac{m(m-1)}2,
\]
and hence
\[
\mathbb E\|\xi_k\|^2
\le\alpha_k^2L^2M_Y^2\frac{m^2(m-1)^2}{4}.
\]
By Lipschitz continuity, Cauchy's inequality and \eqref{10},
\begin{align*}
\|\delta_k\|
&\le L\sum_{\ell=0}^{m-1}\|\bar{\mathbf y}_k^\ell\|,\\
\|\delta_k\|^2
&\le mL^2\sum_{\ell=0}^{m-1}\|\bar{\mathbf y}_k^\ell\|^2
\le mL^2\left(S_m\|\bar{\mathbf y}_k\|^2+c_YT_m\alpha_k^2\right),
\end{align*}
where
\[
S_m:=\sum_{\ell=0}^{m-1}s^\ell,\qquad
T_m:=\sum_{\ell=0}^{m-1}\sum_{r=0}^{\ell-1}s^r.
\]
Therefore,
\[
\mathbb E\|\xi_k+\delta_k\|^2
\le C_e\alpha_k^2+C_b\mathbb E\|\bar{\mathbf y}_k\|^2,
\]
where
\[
C_e:=\frac{L^2M_Y^2m^2(m-1)^2}{2}+2mL^2c_YT_m,
\qquad
C_b:=2mL^2S_m.
\]

Let
\(A_k:=\mathbb E\|\mathbf x_k-\mathbf x_\star\|^2\) and
\(B_k:=\mathbb E\|\bar{\mathbf y}_k\|^2\). Then
\begin{align*}
A_{k+1}
&\le(1+\eta)(1-2m\mu\alpha_k+m^2L^2\alpha_k^2)A_k\\
&\quad+\left(1+\frac1\eta\right)
\alpha_k^2(C_e\alpha_k^2+C_bB_k).
\end{align*}
For constant \(\alpha_k\equiv\alpha\), assume
\(0<\alpha\le\mu/(mL^2)\) and \(0<\alpha<1/(m\mu)\), and choose
\(\eta=m\mu\alpha/[2(1-m\mu\alpha)]\). Then
\[
(1+\eta)(1-m\mu\alpha)=1-\frac{m\mu\alpha}{2},
\qquad
1+\frac1\eta\le\frac{2}{m\mu\alpha}.
\]
Consequently,
\[
A_{k+1}
\le\rho_XA_k+\frac{2C_e}{m\mu}\alpha^3
+\frac{2C_b}{m\mu}\alpha B_k,
\qquad
\rho_X:=1-\frac{m\mu\alpha}{2}.
\]
By Lemma 2,
\[
B_k\le r_Y^kB_0+\frac{D_Y}{1-r_Y}\alpha^2.
\]
Substitution yields
\[
A_{k+1}\le\rho_XA_k+D_A\alpha^3+D_B\alpha r_Y^k,
\]
where
\[
D_A:=\frac{2}{m\mu}
\left(C_e+\frac{C_bD_Y}{1-r_Y}\right),\qquad
D_B:=\frac{2C_bB_0}{m\mu}.
\]
Unrolling over \(K\) epochs gives
\begin{align*}
A_K
&\le\rho_X^KA_0
+D_A\alpha^3\sum_{k=0}^{K-1}\rho_X^{K-1-k}
+D_B\alpha\sum_{k=0}^{K-1}\rho_X^{K-1-k}r_Y^k.
\end{align*}
Since
\[
\sum_{k=0}^{K-1}\rho_X^{K-1-k}
\le\frac{1}{1-\rho_X}=\frac{2}{m\mu\alpha},
\]
and, with \(\bar\rho:=\max\{\rho_X,r_Y\}\),
\[
\sum_{k=0}^{K-1}\rho_X^{K-1-k}r_Y^k
\le K\bar\rho^{K-1},
\]
we obtain
\[
A_K
\le\rho_X^KA_0+\frac{2D_A}{m\mu}\alpha^2
+D_B\alpha K\bar\rho^{K-1}.
\]
Equivalently,
\[
\mathbb E\|\mathbf x_K-\mathbf x_\star\|^2
\le\rho_X^K\mathbb E\|\mathbf x_0-\mathbf x_\star\|^2
+\frac{2D_A}{m\mu}\alpha^2
+D_B\alpha K\bar\rho^{K-1}.
\]
Since \(K\bar\rho^{K-1}\to0\),
\[
\limsup_{K\to\infty}\mathbb E\|\mathbf x_K-\mathbf x_\star\|^2
\le\frac{2D_A}{m\mu}\alpha^2.
\]
In particular,
\[
\mathbb E\|\mathbf x_K-\mathbf x_\star\|^2
=\mathcal O(\rho_X^K)+\mathcal O(K\bar\rho^{K-1})
+\mathcal O(\alpha^2).
\]
\end{proof}

\section{Proof of Theorem 2}\label{Theorem2} 
\begin{proof}
Define
\(A_k:=\mathbb E\|\mathbf x_k-\mathbf x_\star\|^2\) and
\(B_k:=\mathbb E\|\bar{\mathbf y}_k\|^2\). By the same argument as in
Theorem~\ref{thm:partial_constant}, with
\(\eta_k=m\mu\alpha_k/[2(1-m\mu\alpha_k)]\) and
\[
\alpha_k\le\bar\alpha_{\mathrm d}
\le\min\left\{\frac{\mu}{mL^2},\frac{1}{2m\mu}\right\},
\]
there exist constants \(c_x,d_e,d_b>0\), independent of \(k\) and
\(\alpha_k\), such that
\[
A_{k+1}\le(1-c_x\alpha_k)A_k+d_e\alpha_k^3+d_b\alpha_kB_k,
\qquad c_x:=\frac{m\mu}{2}.
\]
Moreover, Lemma~\ref{lem:partial_consensus} gives
\[
B_{k+1}\le r_YB_k+D_Y\alpha_k^2,\qquad0<r_Y<1.
\]

Unrolling the latter recursion gives
\[
B_k\le r_Y^kB_0+D_Y\sum_{t=0}^{k-1}r_Y^{k-1-t}\alpha_t^2.
\]
Clearly, \(r_Y^kB_0\to0\). Since
\(\sum_{k=0}^\infty\alpha_k^2<\infty\), we have \(\alpha_k^2\to0\).
For any \(\varepsilon>0\), choose \(K_0\) such that
\(\alpha_t^2\le\varepsilon\) for all \(t\ge K_0\). Then, for \(k>K_0\),
\begin{align*}
\sum_{t=0}^{k-1}r_Y^{k-1-t}\alpha_t^2
&=\sum_{t=0}^{K_0-1}r_Y^{k-1-t}\alpha_t^2
+\sum_{t=K_0}^{k-1}r_Y^{k-1-t}\alpha_t^2\\
&\le\sum_{t=0}^{K_0-1}r_Y^{k-1-t}\alpha_t^2
+\frac{\varepsilon}{1-r_Y}.
\end{align*}
The first term tends to zero, and \(\varepsilon\) is arbitrary; hence
\(B_k\to0\).

Next,
\begin{align*}
\sum_{k=0}^\infty\alpha_kB_k
&\le B_0\sum_{k=0}^\infty\alpha_kr_Y^k
+D_Y\sum_{k=0}^\infty\alpha_k
\sum_{t=0}^{k-1}r_Y^{k-1-t}\alpha_t^2.
\end{align*}
Since \(\alpha_k\le\bar\alpha_{\mathrm d}\),
\[
\sum_{k=0}^\infty\alpha_kr_Y^k
\le\frac{\bar\alpha_{\mathrm d}}{1-r_Y}<\infty.
\]
Exchanging the order of summation and again using
\(\alpha_k\le\bar\alpha_{\mathrm d}\), we have
\begin{align*}
\sum_{k=0}^\infty\alpha_k
\sum_{t=0}^{k-1}r_Y^{k-1-t}\alpha_t^2
&=\sum_{t=0}^\infty\alpha_t^2
\sum_{k=t+1}^\infty\alpha_kr_Y^{k-1-t}\\
&\le\frac{\bar\alpha_{\mathrm d}}{1-r_Y}
\sum_{t=0}^\infty\alpha_t^2<\infty.
\end{align*}
Therefore, \(\sum_{k=0}^\infty\alpha_kB_k<\infty\).

Since \(\alpha_k\le\bar\alpha_{\mathrm d}\) and
\(\sum_k\alpha_k^2<\infty\),
\[
\sum_{k=0}^\infty\alpha_k^3
\le\bar\alpha_{\mathrm d}\sum_{k=0}^\infty\alpha_k^2<\infty.
\]
Together with \(\sum_k\alpha_kB_k<\infty\), this implies
\[
\sum_{k=0}^\infty
\left(d_e\alpha_k^3+d_b\alpha_kB_k\right)<\infty.
\]
Define \(\varepsilon_k:=d_e\alpha_k^3+d_b\alpha_kB_k\). Then
\[
A_{k+1}\le A_k-c_x\alpha_kA_k+\varepsilon_k,
\qquad
\sum_{k=0}^\infty\varepsilon_k<\infty.
\]
Applying Lemma~\ref{lem:det_rs} with
\(v_k=A_k\), \(b_k=c_x\alpha_kA_k\), and \(c_k=\varepsilon_k\), we conclude
that \(A_k\) converges to a finite limit \(A_\infty\ge0\) and
\(\sum_{k=0}^\infty\alpha_kA_k<\infty\).

Suppose, by contradiction, that \(A_\infty>0\). Then there exists \(K_1\)
such that \(A_k\ge A_\infty/2\) for all \(k\ge K_1\), and hence
\[
\sum_{k=0}^\infty\alpha_kA_k
\ge\frac{A_\infty}{2}\sum_{k=K_1}^\infty\alpha_k=\infty,
\]
contradicting \(\sum_k\alpha_kA_k<\infty\). Thus \(A_\infty=0\), and hence
\(A_k\to0\). Together with \(B_k\to0\), this gives
\[
\lim_{k\to\infty}\mathbb E\|\mathbf x_k-\mathbf x_\star\|^2=0,
\qquad
\lim_{k\to\infty}\mathbb E\|\bar{\mathbf y}_k\|^2=0.
\]
This completes the proof.
\end{proof}
\bibliographystyle{IEEEtran}
\bibliography{autosam.bib}
\end{document}